\documentclass{article}

\RequirePackage{ifpdf}
\ifpdf \usepackage[pdftex]{graphicx} \pdfcompresslevel=9
\else \usepackage[dvips]{graphicx} \fi

\usepackage[a4paper,scale=0.75]{geometry}
\usepackage{amsmath,amssymb,amsthm,gensymb}
\usepackage{graphics,color}
\usepackage{url}

\newtheorem{theorem}{Theorem}
\newtheorem{lemma}{Lemma}
\newtheorem{corollary}{Corollary}

\newtheorem{observation}{Observation}

\theoremstyle{definition}
\newtheorem{definition}{Definition}

\providecommand{\abs}[1]{{\lvert#1\rvert}}

\newcommand{\RR}{\mathbb{R}}
\newcommand{\CC}{\mathbb{C}}

\def\xb{{\bar{x}}}
\def\yb{{\bar{y}}}

\begin{document}

\title{On Divergence-based Distance Functions for \\
Multiply-connected Domains}
\author{Renjie Chen \and Craig Gotsman \and Kai Hormann}
\date{}

\maketitle


\begin{abstract}
Given a finitely-connected bounded planar domain $\Omega$, it is possible to define a {\it divergence distance} $D(x,y)$ from $x\in\Omega$ to $y\in\Omega$, which takes into account the  complex geometry of the domain. This distance function is based on the concept of $f$-divergence, a distance measure traditionally used to measure the difference between two probability distributions. The relevant probability distributions in our case are the Poisson kernels of the domain at $x$ and at $y$. We prove that for the $\chi^2$-divergence distance, the gradient by $x$ of $D$ is opposite in direction to the gradient by $x$ of $G(x,y)$, the Green's function with pole $y$. Since $G$ is harmonic, this implies that $D$, like $G$, has a single extremum in $\Omega$, namely at $y$ where $D$ vanishes. Thus $D$ can be used to trace a gradient-descent path within~$\Omega$ from $x$ to $y$ by following $\nabla_x D(x,y)$, which has significant computational advantages over tracing the gradient of $G$. This result can be used for robotic path-planning in complex geometric environments.
\end{abstract}


\section{Introduction}\label{sec:introduction}

The $f$-divergence function, first introduced by Kullback and Leibler~\cite{Kullback:1951:OIA} and later generalized by Csisz{\'a}r~\cite{Csiszar:1967:ITM}, measures the difference between two probability distributions.

\begin{definition}
Let $f:\RR \rightarrow\RR$ be a strictly convex function with $f(1)=0$ and $p,q\colon E\to[0,1]$ be two non-negative real functions on some domain $E$ such that
\[
  \int_E p(t) dt = \int_E q(t) dt = 1.
\]
We then call
\begin{equation}\label{eq:f-divergence}
  d_f(p,q) := \int_E p(t) f\biggl(\frac{p(t)}{q(t)}\biggr) dt
\end{equation}
the \emph{$f$-divergence} between $p$ and $q$.
\end{definition}

It is well known~\cite{Kullback:1951:OIA} that
\[
  d_f(p,q) \ge 0
\]
and
\[
  d_f(p,q) = 0 \qquad\iff\qquad p = q.
\]
The $f$-divergence is not necessarily a metric, since it is typically not symmetric and does not satisfy the triangle inequality. Many instances of $f$ have been proposed over the years, each suitable for some specific application, mostly in probability theory, statistics, and information theory. The interested reader is referred to Liese and Vajda~\cite{Liese:2006:ODA} for a survey of the possibilities. We just mention the most popular choices of $f$: Kullback-Liebler: $f(t)=-\log t$, Total Variation: $f(t)=|t-1|$, Chi Squared ($\chi^2$): $f(t)=t^2-1$, Hellinger: $f(t)=(\sqrt{t}-1)^2$.

In this context, it is useful to remember the concept of the \emph{dual function} of $f$,
\[
  f^\ast(t) = t f(1/t).
\]
For example, if $f(t)=-\log t$, then $f^\ast(t)=t \log t$, and if $f(t)=\abs{t-1}$, then $f^\ast(t)=f(t)$. It is easy to verify that
\begin{itemize}
  \item $f$ is (strictly) convex if and only if $f^\ast$ is (strictly) convex;
  \item $d_f(p,q)=d_{f^\ast}(q,p)$;
  \item $f(1) \le d_f(p,q) \le f(0)+f^\ast(0)$.
\end{itemize}

Chen et al.~\cite{Chen:2016:PPW} took advantage of the concept of $f$-divergence to define the $f${\it-divergence distance} $D_f(x,y)$ between pairs of points in a planar domain $\Omega$ (see Section~\ref{sec:divergence_distance}) and showed that this distance can play an important role in robotic path-planning applications. In that context, given two points $x,y$ in a geometrically complex domain $\Omega$, it is important to be able to generate a path from $x$ to $y$ which stays completely within $\Omega$. Given a non-negative distance function $D(x,y)$ on the domain, which vanishes only when $x=y$, it is possible to trace such a path by following the gradient of $D$, as long as this gradient does not vanish at some local minimum of $D$. Fig.~\ref{fig:example} (Left) shows some examples of such paths. Chen et al.~\cite{Chen:2016:PPW} show that the $f$-divergence distance $D_f(x,y)$ has this important property in the case of a simply-connected bounded domain. Their proof uses the concept of conformal invariance and relies heavily on the Riemann conformal mapping theorem~\cite{Ahlfors:1979:CAA}, namely, that any simply-connected domain is conformally equivalent to the unit disk. They show that $D_f$ is a conformal invariant of the domain, as is the Green's function $G$, leading to the conclusion that the gradient of the $f$-divergence distance function is opposite in direction to the gradient of the Green's function, independent of $f$. Since the harmonic $G$ has the required property, so does $D_f$.

Although  $D_f$ and $G$ both have the same desirable property and generate identical gradient paths between points, Chen et al.~\cite{Chen:2016:PPW,Chen:2017:PDF} show that there are significant computational advantages to using $D_f$ instead of $G$ in practical path-planning applications, where the domain is discretized.

In this paper we generalize the results of Chen et al.~\cite{Chen:2016:PPW} to the more challenging case of multiply-connected domains. Since these are not conformally equivalent to the unit disk, the proof technique used for the simply-connected domain is no longer applicable. Instead, in our central Theorem \ref{theorem:parallel}, we provide a direct proof for the $\chi^2$-divergence distance using methods of complex analysis (see Section~\ref{sec:chi2}).


\section{The $f$-divergence Distance}\label{sec:divergence_distance}

The $f$-divergence distance between two points in a  bounded planar domain $\Omega$ is defined using the \emph{Poisson kernel} of $\Omega$ at $y\in\Omega$,
\begin{equation}\label{eq:P}
  P(w,y) = - \frac{1}{2\pi} \frac{\partial G}{\partial n(w)} (w,y),
  \qquad w \in \partial\Omega,
\end{equation}
which is the derivative of the Green's function $G(x,y)$ with pole $y$ at the boundary point $w$, in the direction of the unit outer normal $n(w)$~\cite{Garnett:2005:HM}. Thus $P$ is positive in $\Omega$ and for all $y\in\Omega$,
\begin{equation*}
  \oint_{\partial\Omega} P(w,y) \abs{dw} =1
\end{equation*}
where $\abs{dw}$ is the usual arc length differential.
\begin{definition}
Let $f$ be a strictly convex function with $f(1)=0$ and $\Omega$ be a bounded planar domain. The \emph{$f$-divergence distance} from $x\in\Omega$ to $y\in\Omega$ is the contour integral
\begin{equation}\label{eq:df}
  D_f(x,y) :=  d_f(P(\cdot , x), P(\cdot, y))=\oint_{\partial\Omega} P(w,y) f\biggl(\frac{P(w,x)}{P(w,y)}\biggr) \abs{dw},
\end{equation}
\end{definition}
The following Theorem summarizes the basic properties of the $f$-divergence distance.
\begin{theorem}\label{theorem:div_distance}
Let $D_f(x,y)$ be a $f$-divergence distance on a planar domain $\Omega$. Then:
\begin{itemize}
\item $D_f$ is non-negative: $D_f(x,y)\ge 0$.
\item $D_f$ is symmetric: $D_f(x,y)=D_f(y,x)$.
\item $D_f$ is constant on the boundary $\partial\Omega$ (unless it is infinite there).
\item $D_f$ is strictly subharmonic: for any $y\in\Omega$, $\nabla_x^2 D_f(x,y)>0$.
\end{itemize}
\end{theorem}
\begin{proof}
See Chen et al. \cite{Chen:2016:PPW}.
\end{proof}

The second property is a little surprising, since in general the $f$-divergence of two probability functions (\ref{eq:f-divergence}) is not symmetric. Despite this, the $f$-divergence distance will typically not be a metric, as it does not satisfy the triangle inequality. A rare exception is the Total Variation divergence distance
\[
  D_{\mathrm{TV}}(x,y) = \oint_{\partial\Omega} \abs{P(w,x)-P(w,y)} \abs{dw}.
\]
Relying on the symmetry of $D_f$ and the properties of the dual function $f^\ast$, we have:
\begin{corollary}
The functions $f$ and $f^\ast$ generate \emph{identical} divergence distances,
\[
  D_f(x,y) = D_{f^\ast}(x,y).
\]
\end{corollary}

\section{Preliminaries}\label{sec:preliminaries}

To prove our central Theorem \ref{theorem:parallel}, we need a few preliminaries about Green's functions and their consequences on divergence distance functions.

Let $\Omega$ be an open, bounded, and finitely-connected domain. Following~\cite{Garnett:2005:HM}, we consider $\Omega$ as a subset of $\CC$ and define for any $y\in\Omega$ the \emph{Green's function with pole at $y$} as
\begin{equation}\label{eq:g}
  G(x,y) = \log \frac{1}{\abs{x-y}} + H(x,y),
  \qquad x \in \bar{\Omega}\setminus\{y\},
\end{equation}
where $H(x,y)$ is the harmonic solution of the Dirichlet problem for the boundary values $f(w)=\log\abs{w-y}$, $w\in\partial\Omega$. Thus $G$ vanishes on $\partial\Omega$ and is positive and harmonic on $\Omega \setminus\{ y \}$. Furthermore, $G$ is symmetric: $G(x,y)=G(y,x)$. We may treat $x$ and $y$ as complex numbers and, taking advantage of the convenient complex algebra, use the subscripts $x$, $y$, $\xb$, and $\yb$ to denote the Wirtinger derivatives of $G(x,y)$ with respect to its first and second variable and their conjugates, for example,
\[
  G_x = \frac{\partial G}{\partial x}, \qquad
  G_\xb = \frac{\partial G}{\partial \xb}, \qquad
  G_{x\yb} = \frac{\partial^2 G}{\partial x \partial \yb},\qquad
  G_{\xb y} = \frac{\partial^2 G}{\partial \xb \partial y},
\]
and likewise for $H$. Note that according to this notation, we have, for example,
\[
  \frac{\partial}{\partial w} G(w,y) = G_x(w,y), \qquad
  \nabla_x G(w,x) = 2 \frac{\partial G}{\partial \xb} \qquad
  \nabla_y G_x(w,y) = 2 \frac{\partial}{\partial \yb} G_x(w,y) = 2 G_{x\yb}(w,y).
\]
We further say that two complex numbers $a,b\in\CC$ are \emph{proportional} and write $a\propto b$, if $\arg a=\arg b$, namely $a/b\in\RR$ with $a/b>0$.

\begin{observation}\label{observation:Poisson/Green-relation}
The Poisson kernel~\eqref{eq:P} satisfies
\[
  P(w,y)
  = \frac{1}{2\pi} \abs{\nabla_x G(w,y)}
  = \frac{1}{\pi} \abs{G_\xb(w,y)}
  = \frac{1}{\pi} \abs{G_x(w,y)},
\]
because Green's function is real, positive in $\Omega$, and vanishes at $\partial\Omega$, hence its gradient at $w\in\partial\Omega$ points in the opposite direction of $n(w)$, the unit inward normal at $w$.
\end{observation}

\begin{observation}\label{observation:Poisson-quotient}
For any $w\in\partial\Omega$ and $y,z\in\Omega$ with $y\ne z$, we have
\[
  \frac{P(w,y)}{P(w,z)}
  = \biggl\lvert \frac{G_x(w,y)}{G_x(w,z)} \biggr\rvert
  = \frac{G_x(w,y)}{G_x(w,z)}
\]
by Observation~\ref{observation:Poisson/Green-relation} and because $G_x(w,y)\propto G_x(w,z)$.
\end{observation}

Using these observations, we can express the $f$-divergence distance as a complex contour integral in terms of $G_x$.

\begin{lemma}\label{lemma:df-using-gx}
The $f$-divergence distance \eqref{eq:df} satisfies
\[
  D_f(x,y) = \frac{i}{\pi} \oint_{\partial\Omega} f\biggl(\frac{G_x(w,x)}{G_x(w,y)}\biggr) G_x(w,y) dw.
\]
\end{lemma}
\begin{proof}
By Observation~\ref{observation:Poisson/Green-relation},
\[
  D_f(x,y)
  = \frac{1}{\pi}
      \oint_{\partial\Omega} f\biggl(\frac{P(w,x)}{P(w,y)}\biggr) \abs{G_\xb(w,y)} \abs{dw}
  = \frac{1}{\pi}
      \oint_{\partial\Omega} f\biggl(\frac{P(w,x)}{P(w,y)}\biggr) \abs{G_\xb(w,x) dw}.
\]
Since $G_\xb(w,y)$ and $dw$ are orthogonal to each other, with $G_\xb(w,y)$ pointing 90\degree\ to the left of $dw$,
we have
\begin{equation}\label{eq:path-to-contour}
  \abs{G_\xb(w,y) dw} = i \overline{G_\xb(w,y)} dw = i G_x(w,y) dw,
\end{equation}
resulting in
\[
  D_f(x,y)
  = \frac{i}{\pi} \oint_{\partial\Omega} f\biggl(\frac{P(w,x)}{P(w,y)}\biggr) G_x(w,y) dw.
\]
Note that the complex integral is computed counter-clockwise for the outer boundary loop and clockwise for the internal boundary loops.
The Lemma then follows by Observation~\ref{observation:Poisson-quotient}.
\end{proof}

\begin{observation}\label{observation:Green-holomorphic}
Since $G(x,y)$ is harmonic in $x$ and $y$, $G_x$ is a function of $x$ and not of $\xb$. For the same reason, $G_{x\yb}$ is a function of only $x$ and $\yb$. As a result, $G_x$ and $G_{x\yb}$ are both holomorphic in $x$, except at possible poles.
\end{observation}

\begin{lemma}\label{lemma:gx-proportional-1/gxb}
For any $x,y\in\Omega$ with $x\ne y$,
\[
  G_x(x,y) \propto \frac{1}{G_\xb(x,y)}.
\]
\end{lemma}
\begin{proof}
Since $G$ is real, we have $G_\xb=\overline{G_x}$, and the Lemma then follows, because $a\propto 1/\bar{a}$ for any $a\in\CC\setminus\{0\}$.
\end{proof}

\begin{lemma}\label{lemma:limit}
For any $y,z\in\Omega$ with $y\ne z$,
\[
  \lim_{x\to y} (x-y) \frac{G_x(x,y)}{G_x(x,z)} = -\frac{1}{2G_x(y,z)}
\]
\end{lemma}
\begin{proof}
It follows from~\eqref{eq:g} that
\begin{equation}\label{eq:gx}
  G_x(x,y) = -\frac{1}{2(x-y)} + H_x(x,y).
\end{equation}
Therefore,
\[
  \lim_{x\to y} (x-y) G_x(x,y) = -\frac12,
\]
because $H$ is twice continuously differentiable and thus $H_x$ is finite. The Lemma follows.
\end{proof}

\begin{lemma}\label{lemma:gxyb-real}
For any $x\in\Omega$,
\[
  G_{x\yb}(x,x)
  \in \RR.
\]
\end{lemma}
\begin{proof}
It follows from~\eqref{eq:gx} that
\begin{equation}\label{eq:gxyb=hxyb}
  G_{x\yb}(x,y) = H_{x\yb}(x,y),
\end{equation}
and since $H$ is twice continuously differentiable, $G_{x\yb}(x,x)$ exists, even though $G(x,x)$ does not. By the symmetry of $G$, we further have
\[
  \overline{G_{x\yb}(x,x)} = G_{\xb y}(x,x) = G_{x\yb}(x,x),
\]
implying the Lemma.
\end{proof}


\section{The $\chi^2$-divergence Distance}\label{sec:chi2}

Let us now focus on the \emph{$\chi^2$-divergence distance}, which is given by the function $f(t)=t^2-1$, and denote it by $D(x,y) = D_f(x,y)$. Note that $D$ is infinite on the boundary $\partial\Omega$, but this does not pose a problem. We first express $D$ in terms of the derivatives of $G$ and $H$, mentioned in \eqref{eq:g}.

\begin{theorem}\label{theorem:chi2-as-g-and-h}
The $\chi^2$-divergence distance can be written as
\begin{equation}\label{eq:distance-as-g-and-h}
  D(x,y) = 2 \frac{H_x(x,x)}{G_x(x,y)} + \frac12 \frac{G_{xx}(x,y)}{G_x(x,y)^2} - 1.
\end{equation}
\end{theorem}
\begin{proof}
By Lemma~\ref{lemma:df-using-gx},
\[
  D(x,y) = \frac{i}{\pi} \oint_{\partial\Omega} \frac{G_x(w,x)^2}{G_x(w,y)} dw
         - \frac{i}{\pi} \oint_{\partial\Omega} G_x(w,y) dw.
\]
According to Observation~\ref{observation:Green-holomorphic}, both integrands are holomorphic. For the second integrand, which has a first-order pole in $\Omega$ at $w=y$, applying Cauchy's Residue Theorem \cite{Ahlfors:1979:CAA} and using Lemma~\ref{lemma:limit}, we get
\[
  \oint_{\partial\Omega} G_x(w,y) dw = 2\pi i \lim_{w\to y} (w-y) G_x(w,y) = -\pi i.
\]
The first integrand has a second-order pole at $w=x$, and so by Cauchy's Residue Theorem and Lemma~\ref{lemma:limit},
\begin{align*}
  \oint_{\partial\Omega} \frac{G_x(w,x)^2}{G_x(w,y)} dw
  &= 2\pi i \lim_{w\to x} \frac{d}{dw} \biggl( {(w-x)}^2 \frac{G_x(w,x)^2}{G_x(w,y)} \biggr)\\
  &= 2\pi i \lim_{w\to x}
       \biggl( 2 (w-x) G_x(w,x) \frac{G_x(w,x)+(w-x) G_{xx}(w,x)}{G_x(w,y)}
               - \bigl[ (w-x) G_x(w,x) \bigr]^2 \frac{G_{xx}(w,y)}{G_x(w,y)^2} \biggr)\\
  &= -2\pi i
       \biggl( \frac{\lim_{w\to x} \bigl( G_x(w,x) + (w-x) G_{xx}(w,x) \bigr)}{G_x(x,y)}
               + \frac{G_{xx}(x,y)}{4 G_x(x,y)^2} \biggr).
\end{align*}
Since it follows from~\eqref{eq:g} that
\[
  G_x(x,y) = -\frac{1}{2(x-y)} + H_x(x,y)
\]
and
\[
  G_{xx}(x,y) = \frac{1}{2{(x-y)}^2} + H_{xx}(x,y),
\]
we have
\[
  \lim_{w\to x} \bigl( G_x(w,x) + (w-x) G_{xx}(w,x) \bigr)
  = \lim_{w\to x} \bigl( H_x(w,x) + (w-x) H_{xx}(w,x) \bigr)
  = H_x(x,x),
\]
and the Theorem follows.
\end{proof}

\begin{figure}[t!]\centering
  \includegraphics[width=0.9\columnwidth]{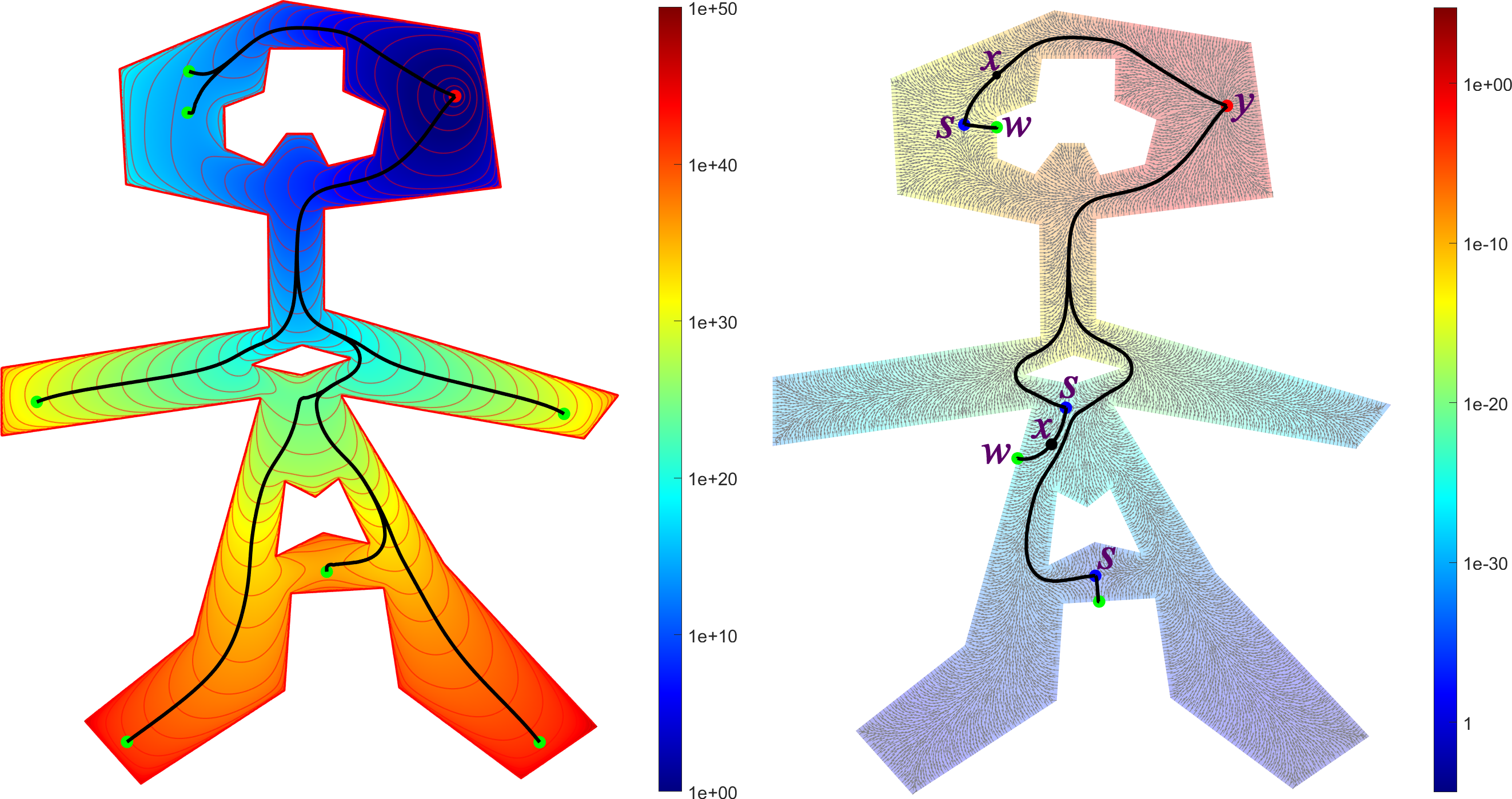}
  \caption{(Left) Gradient-descent paths of the $\chi^2$ divergence distance from various source points (green) to a common target point (red). The domain is color-coded according to the distance. The red curves are contours (level sets) of the distance function.
  (Right) The integral path $\gamma$ is a sub-path of the gradient-descent path of the Green's function from target $y$ to $x$, continuing further to the boundary point $w$, possibly passing through a saddle point $s$. The gradient of Green's function is visualized as a unit-length direction field, while the magnitude of the gradient is the color-coded background.
   \label{fig:example}}
\end{figure}

\begin{theorem}\label{theorem:parallel}
The gradient of the $\chi^2$-divergence distance to $y$ is proportional to the negative gradient of the Green's function with pole $y$, that is,
\[
  \nabla_x D(x,y) \propto -\nabla_x G(x,y)
\]
\end{theorem}

\begin{proof}
Since $G$ is harmonic, we have
\[
  \nabla_x G_x(x,y) = 2 G_{x\xb}(x,y) = 0, \qquad\text{and}\qquad
  \nabla_x G_{xx}(x,y) = \nabla_{\xb} G_{x\xb}(x,y) = 0,
\]
The harmonicity of $H$, together with~\eqref{eq:gxyb=hxyb}, implies
\[
  \nabla_x H_x(x,x)
  = 2 \frac{\partial}{\partial \xb} H_x(x,x)
  = 2 H_{x\xb}(x,x) + 2 H_{x\yb}(x,x)
  = 2 H_{x\yb}(x,x)
  = 2 G_{x\yb}(x,x).
\]
We then conclude from Theorem~\ref{theorem:chi2-as-g-and-h} that
\begin{align*}
  \nabla_x D(x,y)
  &=   2 \frac{\nabla_x H_x(x,x)}{G_x(x,y)}
     - 2 \frac{H_x(x,x) \nabla_x G_x(x,y)}{G_x(x,y)^2}
     + \frac12 \frac{\nabla_x G_{xx}(x,y)}{G_x(x,y)^2}
     - \frac{G_{xx}(x,y) \nabla_x G_x(x,y)}{G_x(x,y)^3}\\
  &= 4\frac{G_{x\yb}(x,x)}{G_x(x,y)}  = 2\frac{G_{x \yb}(x,x)}{|G_x(x,y)^2|} \nabla_x G(x,y)
\end{align*}
By Lemma~\ref{lemma:gxyb-real}, $G_{x \yb}(x,x)$ is real, so it remains to show that the ratio
\[
  \phi(x) = \phi(x,y) = \frac{\nabla_x G(x,y)}{\nabla_x D(x,y)}
\]
is also negative for any $x\in\Omega$ with $x\ne y$.

To keep the notation simple, let us omit the subscript $x$ and tacitly assume that all gradients are taken with respect to $x$. Following this convention, we then have by the harmonicity of $G$ that
\begin{align*}
  0
  &= \nabla^2 G(x,y)
   = \nabla \cdot \bigl(\phi(x) \nabla D(x,y) \bigr)\\
  &= \nabla \phi(x) \cdot \nabla D(x,y) + \phi(x) \nabla^2 D(x,y)\\
  &= \frac{\nabla \phi(x)}{\phi(x)} \cdot \nabla G(x,y) + \phi(x) \nabla^2 D(x,y),
\end{align*}
implying
\begin{equation}\label{eq:gradient-product-negative}
  \nabla \phi(x) \cdot \nabla G(x,y)
   =  - \phi(x)^2 \nabla^2 D(x,y)
  \le 0,
\end{equation}
because $D$ is subharmonic. Let us now consider the unique gradient-descent path, defined by $\nabla G$, from $y$ to $x$ and further on to some boundary point $w\in\partial\Omega$, and denote by $\gamma$ the reverse sub-path from $w$ to $x$ with vector differential $dz$. By definition of $\gamma$, we have $dz/\abs{dz}=\nabla G(z,y)/\abs{\nabla G(z,y)}$, and according to the Gradient Theorem, we then have
\begin{equation}\label{eq:path-integral}
  \phi(x) - \phi(w)
  = \int_\gamma \nabla \phi(z) \cdot dz
  = \int_\gamma \frac{\nabla \phi(z) \cdot \nabla G(z,y)}{\abs{\nabla G(z,y)}} \abs{dz} \le 0,
\end{equation}
where the last inequality follows from~\eqref{eq:gradient-product-negative}.

As $\nabla G(w,y)$ is normal to the boundary of $\Omega$ at $w$, so is $\nabla D(w,y)$ by Theorem~\ref{theorem:parallel}. But while $\nabla G(w,y)$ points \emph{inwards}, $\nabla D(w,y)$ points \emph{outwards}, because $D$ is subharmonic and hence obtains its maximum on $\partial\Omega$. Consequently, $\phi(w) < 0$, and it follows from~\eqref{eq:path-integral} that
\[
  \phi(x) \le \phi(w) < 0.
\]

Note that for multiply-connected domains, it may happen that the gradient-descent path of $G$ from $y$ to $w$ encounters a saddle point $s$ of $G$. In this case, even though $\nabla G(s)=0$, there still exists a direction of steepest descent, given by the principal curvature directions (the eigenvectors of the Hessian of $G$) at that point, and we follow this direction to extend the path beyond $s$, see the example in Fig.~\ref{fig:example} (Right). Since the number of saddle points is finite, this strategy guarantees that the path eventually terminates at some $w\in\partial\Omega$. Hence, the path $\gamma$ from $w$ to $x$ in~\eqref{eq:path-integral} is well-defined and consequently $\phi(x)<0$, even in this case.
\end{proof}

\section{Conclusion}\label{sec:conclusion}
We have shown that the $\chi^2$-divergence distance of a multiply-connected domain behaves similarly to the Green's function of that domain, namely that it has no local minima, except at the target point, where it has a global minimum of zero. This implies that the divergence distance function may be used to trace paths between pairs of points in the domain by gradient descent. Our numerical experiments have indicated that all other divergence distances (based on other strictly convex $f$'s) also have this desirable property, although we have been unable to prove it. The obstacle is applying Cauchy's Residue Theorem to compute the contour integral in the proof of Theorem \ref{theorem:parallel}. For many $f$'s, the resulting integrand may have branches or higher-order poles, significantly complicating the computation of that integral.


\bibliographystyle{abbrv}
\bibliography{journals,references}

\end{document}